\documentclass{amsart}
\title[On Cayley numbers]{Cayley numbers with arbitrarily many distinct prime factors}

\author[T. Dobson]{Ted Dobson}
\address{Ted Dobson, Mississippi State University,  \newline
PO Drawer MA, Mississippi State, MS 39762, USA, and \newline
University of Primorska, IAM, Muzejeska trg 2, 6000 Koper, Slovenia}
\email{dobson@math.msstate.edu}

\author[P. Spiga]{Pablo Spiga}
\address{Pablo Spiga, Departimento di Matematica Pura e Applicata,\newline
 University of Milano-Bicocca, Via Cozzi 55, 20126 Milano, Italy.}
\email{pablo.spiga@unimib.it}

\usepackage{amssymb,latexsym,amsthm,amssymb,amsfonts,amsmath,amsopn,amstext,amscd,latexsym}

\newtheorem{thrm}{Theorem}[section]
\newtheorem{lemma}[thrm]{Lemma}
\newtheorem{cor}[thrm]{Corollary}
\newtheorem{prob}[thrm]{Problem}

\numberwithin{equation}{section}
\usepackage{cite}
\usepackage{pdfsync}
\usepackage{color}
\newcommand{\Aut}{\mathop{\textrm{Aut}}}
\newcommand{\Cay}{\mathop{\textrm{Cay}}}
\def\PSL{{\rm PSL}}
\def\soc{{\rm soc}}

\def\cal{\mathcal}

\usepackage{hyperref}

\begin{document}


\begin{abstract}
A positive integer $n$ is a Cayley number if every vertex-transitive graph of order $n$ is  a Cayley graph.  In 1983, Dragan Maru\v si\v c posed the problem of determining the Cayley numbers.  In this paper we give an infinite set $S$ of primes such that every finite product of distinct elements from $S$ is a Cayley number.  This answers a 1996 outstanding question of Brendan McKay and Cheryl Praeger, which they ``believe to be the key unresolved question" on Cayley numbers.
 We also show that, for every finite product $n$  of distinct elements from $S$, every transitive group of degree $n$ contains a semiregular element.
\end{abstract}

\maketitle

\section{Introduction}\label{sec:intro}

In this paper, all groups considered are finite and all graphs and digraphs are finite and have no multiple edges or arcs. (They may have loops and they may be disconnected.) A (di)graph $\Gamma$ is {\bf vertex-transitive} if the automorphism group $\Aut(\Gamma)$ of $\Gamma$ acts transitively on the vertices of $\Gamma$. Let $R$ be a group and let $S$ be a subset of $R$, the {\bf Cayley digraph} on $R$ with connection set $S$, denoted $\Cay(R,S)$, is the digraph with vertex-set $R$ and with $(g,h)$ being an arc if and only if $gh^{-1}\in S$. It is easy to see that $\Cay(R,S)$ is a graph if and only if $S$ is inverse-closed, that is, $S=\{s^{-1}\mid s\in S\}$. It is also clear that the action of $R$ on itself by right multiplication gives rise to a group of automorphisms of $\Cay(R,S)$ which is transitive on its vertex-set; thus $\Cay(R,S)$ is vertex-transitive. With a slight abuse of terminology, we say that a (di)graph $\Gamma$ is a Cayley (di)graph, if $\Gamma$ is isomorphic to some Cayley (di)graph: it is well-known and easy to prove~\cite[Lemma~$4$]{Sabidussi1958} that this happens if and only if $\Aut(\Gamma)$ contains  a subgroup $R$ with $R$ transitive on the vertices of $\Gamma$ and with the identity being the only element of $R$ fixing some vertex of $\Gamma$.

A positive integer $n$ is called a {\bf Cayley number} if every vertex-transitive graph of order $n$ is a Cayley graph.
In 1983, Dragan Maru\v si\v c~\cite{Marusic1983} asked for a concrete determination of the positive integers $n$ for which every vertex-transitive graph of order $n$ is a Cayley graph, that is, Maru\v si\v c posed the problem of determining Cayley numbers. Clearly, every prime number is a Cayley number, and $10$ is not a Cayley number because the Petersen graph is not a Cayley graph.

The question of Maru{\v{s}}i{\v{c}} has generated a fair amount of interest, with most  results giving explicit numbers that are not Cayley.  In fact, it seems to be easier to provide integers that are not Cayley numbers: to show that $n$ is not a Cayley number it suffices to exhibit one single vertex-transitive graph of order $n$ that is not a Cayley graph, but to show that $n$ is a Cayley number one needs to prove that each vertex-transitive graph of order $n$ is a Cayley graph.

It is useful to observe that the integers which are not Cayley numbers, called {\bf non-Cayley numbers}, are closed under multiplication, that is, if $n$ is a non-Cayley number and $m$ is a positive integer, then $nm$ is also a non-Cayley number.  This is easy to see as, if $\Gamma$ is a graph of order $n$ that is not a Cayley graph, then for every positive integer $m$ the disjoint union of $m$ copies of $\Gamma$ yields a graph of order $nm$ that is also not  a Cayley graph.  The reader is referred to~\cite{McKayP1994,McKayP1996,Seress1998,LiS2005} for many examples of non-Cayley numbers and to~\cite{Dobson2000,Dobson2008,Marusic1983,HassaniIP1998,IranmaneshP2001} for work towards establishing that certain integers are Cayley numbers.


In this paper, we answer a question of Brendan McKay and Cheryl Praeger~\cite[Question]{McKayP1996} that they ``believe to be the key unresolved question'' concerning the structure of the set of Cayley numbers; namely, there exist Cayley numbers that are product of arbitrarily many distinct primes.

\begin{thrm}\label{main1}There exists an infinite set $S$ of primes such that for every finite product $n$ of distinct elements from $S$, every vertex-transitive digraph of order $n$ is a Cayley digraph. Consequently, $n$ is a Cayley number and  there exist
Cayley numbers that are a product of arbitrarily many distinct primes.
\end{thrm}

The proof of Theorem~\ref{main1} follows from Theorem~\ref{maintool}, which we believe to be of independent interest. (A transitive permutation group $G$ is {\bf quasiprimitive} if every non-identity normal subgroup of $G$ is transitive.)

\begin{thrm}\label{maintool}
There exists an infinite set of primes $S$ such that, for every finite product $n$ of distinct elements from $S$ with $n$ not prime, the only quasiprimitive groups of degree $n$ are the alternating and the symmetric groups.
\end{thrm}

Theorem~\ref{maintool} combined with~\cite[Theorem~4.7]{DobsonM2011} gives the following interesting corollary.

\begin{cor}\label{cor}
There exists an infinite set of primes $S$ such that, for every finite product $n$ of distinct  elements from $S$, every transitive group of degree $n$ contains a transitive solvable subgroup.
\end{cor}

These results can be used to obtain a new contribution to elusive groups, see Corollary~\ref{cor2}.

\subsection{Structure of the paper}In Section~\ref{sec:numbertheory} we collect our main tools, namely, a number theoretic result and  a few group theoretic lemmas. We then prove Theorems~\ref{main1} and~\ref{maintool} and Corollary~\ref{cor} in Section~\ref{sec:thrmsproof}. We conclude the paper with some comments and problems in Section~\ref{sec:comments}.

\section{Basics}\label{sec:numbertheory}
\subsection{Number theory}

We start by proving a number theoretic lemma, which depends on Dirichlet's Theorem on primes in arithmetic progressions and on the Prime Number Theorem. (We denote by $\phi(m)$ the Euler's totien function.)

\begin{lemma}\label{ntlemma}
There exists an infinite set $S$ of prime numbers such that, for each $\kappa\in\mathbb{N}$ and for each $p_1,\ldots,p_\kappa\in S$ with $p_1<p_2<\cdots<p_\kappa$, we have
\begin{description}
\item[(i)]$(p_1p_2\cdots p_{\kappa-1})^4<p_\kappa$,
\item[(ii)]$\gcd(p_i,p_{j}-1)=1$ for every $i,j\in \{1,\ldots,\kappa\}$,
\item[(iii)]the product $p_1p_2\cdots p_\kappa$ is not equal to $\frac{q^d-1}{q-1}$, for any prime power $q$  and any $d\geq 1$.
\end{description}
\end{lemma}
\begin{proof}
We construct the set $S$ inductively: we show that
 there exists an infinite family of sets of prime numbers $\{S_\ell\}_{\ell\in\mathbb{N}\setminus\{0\}}$  such that, for every $\ell\in\mathbb{N}\setminus\{0\}$:
\begin{enumerate}
\item[(a)] $|S_\ell|=\ell$,
\item[(b)] $S_\ell\subseteq S_{\ell+1}$,
\item[(c)] for each $\kappa\in \{1,\ldots,\ell\}$ and for each $p_1,\ldots,p_\kappa\in S_\ell$ with $p_1<p_2<\cdots<p_\kappa$, the primes $p_1,\ldots,p_\kappa$ satisfy the conditions {\bf (i),~(ii)} and~{\bf (iii)}.
\end{enumerate}Then the proof follows immediately by taking $S:=\bigcup_{\ell\in \mathbb{N}\setminus\{0\}}S_\ell$.

Define $S_1:=\{11\}$ and observe that $11$ is not of the form $(q^d-1)/(q-1)$, for any prime power $q$ and any $d\geq 1$. Now, suppose that $\ell\in \mathbb{N}\setminus\{0\}$, and that $S_\ell=\{p_1,\ldots,p_\ell\}$ has been defined. We construct $S_{\ell+1}$ by carefully choosing a suitable prime $p_{\ell+1}$ and by setting $S_{\ell+1}:=S_\ell\cup\{p_{\ell+1}\}$.

Set $m:=p_1\cdots p_\ell.$
For $N\in \mathbb{N}\setminus\{0\}$, define
\begin{align*}
&\mathcal{T}_N:=\{x\in\mathbb{N}\mid x>m^4,  x \textrm{ prime, }x\leq N \textrm{ and }x\equiv -1\pmod {m} \},\\
&\mathcal{S}_N:=\left\{x=\frac{q^d-1}{q-1}\mid \textrm{ for some prime power }q\textrm{ and some } d\geq 1, x \textrm{ odd,  }x\leq mN \right\}.
\end{align*}

We claim that
\begin{equation}\label{eq:3}
|\mathcal{T}_N|\sim \frac{N}{\phi(m)\log(N)}.
\end{equation}
(Here, given a set $X$ we denote by $|X|$ its cardinality, moreover, by abuse of notation, we let $|\mathcal{T}_N|$ denote the function in the variable $N$ given by $N\mapsto |\mathcal{T}_N|$. Finally, $\sim$ denotes Landau's notation for two asymptotic functions in the variable $N$.) Write
$$
\mathcal{T}'_N:=
\{
x\in\mathbb{N}\mid x \textrm{ prime, }x\leq N \textrm{ and }x\equiv -1\pmod m\}.
$$
As there are at most $m^4$ primes less than or equal to $m^4$, we get
$|\mathcal{T}_N|\geq |\mathcal{T}'_{N}|-m^4.$
From~\cite[Part~II, \S 4, Theorem~$2$]{Serre} and the Prime Number Theorem, we have
$$\lim_{N\to\infty}\frac{|\mathcal{T}'_N|}{N/(\phi(m)\log (N))}=1$$
and Eq.~\eqref{eq:3} is proven.

We claim that
\begin{equation}\label{eq:1}|\mathcal{S}_N|\leq (1+\log_2(mN))(mN)^{1/2}+\log_2(mN)+1.
\end{equation}
   Let $x\in \mathcal{S}_N$ with $x=(q^{d}-1)/(q-1)$ for some prime power $q$ and some $d\geq 1$. Now, $q^{d-1}\leq (q^d-1)/(q-1)=x\leq mN$ and hence $d-1=\log_q(q^{d-1})\leq \log_q(mN)\leq \log_2(mN)$. It follows that $d\leq 1+\log_2(mN)$, that is, we have at most $1+\log_2(mN)$ choices for $d$.

When $d\geq 3$, from $q^{d-1} \leq mN$, we get $q\leq (mN)^{1/2}$ and hence we have at most $(mN)^{1/2}$ choices for $q$. This shows that we have at most $(1+\log_2(mN))(mN)^{1/2}$ choices for $x$ when $d\geq 3$: this accounts for the first summand in the right-hand-side of Eq.~\eqref{eq:1}. Suppose that $d=2$. Then $x=q+1$ is odd and hence $q=2^t$, for some $t\geq 1$. Now, $2^t<x\leq mN$ and hence $t\leq \log_2(mN)$. Thus we have at most $\log_2(mN)$ choices for $x$ in this case. Clearly, we have only one element when $d=1$.

Using Landau's notation, Eq.~\eqref{eq:1} yields
\begin{equation}\label{eq:2}
|\mathcal{S}_N|\in o\left(N/(\phi(m)\log (N)\right)).
\end{equation}

Let $N_\ell$ be the smallest positive integer with
$0<|\mathcal{S}_{N_\ell}|<|\mathcal{T}_{N_\ell}|$ and observe that $N_\ell$ is well-defined from Eqs.~\eqref{eq:2} and~\eqref{eq:3}.

Observe that  every element $x\in\mathcal{T}_{N_\ell}$ is prime and, together with $p_1,\ldots,p_\ell$, satisfies {\bf (i)} and {\bf (ii)}. In fact, if $x\in \mathcal{T}_{N_\ell}$, then $x>m^4$ and $x\equiv -1\pmod m$. Hence, for every $i\in \{1,\ldots,\ell\}$, $x\equiv -1\pmod {p_i}$ and so $\gcd(x,p_i-1)=1$. Suppose that, for every $y\in \mathcal{T}_{N_\ell}$, the set $\{p_1,\ldots,p_\ell,y\}$ fails to satisfy {\bf (iii)}. Then, for every $y\in\mathcal{T}_{N_\ell}$, there exist $i_1,\ldots,i_t\in \{1,\ldots,\ell\}$ with
$$p_{i_1}\cdots p_{i_t}\cdot y=\frac{q^d-1}{q-1},$$ for some prime power $q$ and some $d\geq 1$. Now,
$$p_{i_1}\cdots p_{i_t}\cdot y\leq p_1\cdots p_\ell \cdot y=my\leq mN_\ell$$
and $p_{i_1}\cdots p_{i_t}\cdot y\in \mathcal{S}_{N_\ell}$. Therefore $|\mathcal{S}_{N_\ell}|\geq |\mathcal{T}_{N_\ell}|$, which is a contradiction. Thus there exists $p_{\ell+1}\in \mathcal{T}_{N_\ell}$ such that $\{p_1,\ldots,p_\ell,p_{\ell+1}\}$ satisfies {\bf (i)}, {\bf (ii)} and {\bf (iii)}.
\end{proof}

\subsection{Group theory}\label{sec:gt}

Let $G$ be a transitive permutation group on a set $X$, and let $B\subseteq X$ with $B\neq\emptyset$.  We say $B$ is a {\bf block for $G$} if $B^g=B$ or $B^g\cap B = \emptyset$ for every $g\in G$.  The whole set $X$ and the singleton sets $\{x\}$ are always blocks for $G$ and so are called {\bf trivial blocks}.  If $G$ has a non-trivial block, then $G$ is {\bf imprimitive}, and {\bf primitive} otherwise.  If $B$ is a block for $G$, then $\{B^g\mid g\in G\}$ is a {$G$-invariant partition} of $X$ or system of imprimitivity for $G$.

In the following lemma we use a result of Liebeck and Saxl~\cite[Corollary~$1$]{LiebeckS1985a}: this was one of the first cases where the classification of the finite simple groups had made possible the solution of an outstanding classical problem in the theory of permutation groups. Namely, the classification of the primitive groups of degree $mp$ with $p$ prime and $m<p$. (Here $\soc(G)$ denotes the socle of the group $G$.)

\begin{lemma}\label{uniprimitiveorder}
Let $G$ be a primitive permutation group of degree $n$ with $n$ not prime and let  $p$ be the largest prime divisor of $n$. Then one of the following holds:
\begin{description}
\item[(i)] $p < (n/p)^4$;
\item[(ii)] $\soc(G)=\PSL(d,q)$, $n=\frac{q^d-1}{q-1}$ and $G$ is endowed with one of its natural $2$-transitive actions;
\item[(iii)] $\soc(G)=A_n$ and $G$ is endowed  with its natural $2$-transitive action.
\end{description}
\end{lemma}

\begin{proof}
We suppose that $p\geq (n/p)^4$. Observe that $p\geq n^{4/5}>n^{1/2}$ because $n>1$. In particular, $n=mp$ with $1 < m < p$ and hence the work of Liebeck and Saxl~\cite{LiebeckS1985a} applies.

Let $T$ be the socle of $G$. Now, $T$, $n$ and some information on $p$ are given in the first, second and fourth columns of~\cite[Table~$3$]{LiebeckS1985a}. We  assume that Cases~{\bf (ii)} and~{\bf (iii)} do not occur. With a careful case-by-case examination of each of the remaining lines of~\cite[Table~$3$]{LiebeckS1985a}  we obtain that no primitive group may arise: all computations are straightforward and here, to give an idea of the proof, we deal only with  the fifth row of~\cite[Table~$3$]{LiebeckS1985a} (all other rows are similar).

We have $S=\PSL(d,q)$, $d\geq 4$, $n=(q^d-1)(q^{d-1}-1)/((q^2-1)(q-1))$ and $p=(q^{d-1}-1)/(q-1)$ or $p$ divides $(q^d-1)/(q-1)$. If $p=(q^{d-1}-1)/(q-1)$, we have $p<(n/p)^4$, and hence $p$ must  divide $(q^d-1)/(q-1)$. In particular,
$$\left(\frac{q^{d-1}-1}{q^2-1}\right)^4\leq \left(\frac{n}{p}\right)^4\leq p\leq \frac{q^d-1}{q-1}.$$
However, this inequality is never satisfied when $d\geq 4$.
\end{proof}

In Lemma~\ref{uniprimitiveorder}, the reader should not take too seriously the exponent $4$ in $p<(n/p)^4$. With more effort, one can prove a similar lemma with a smaller exponent; however, for the scope of this paper we content ourselves with our weaker statement.

\begin{lemma}\label{primitiveAn}
Let $S$ be as in Lemma~$\ref{ntlemma}$ and let $n$ be a finite product of distinct elements from $S$ with $n$ not prime.  Then the only primitive groups of degree $n$ are $A_n$ and $S_n$.
\end{lemma}

\begin{proof}
Let $G$ be a primitive group of degree $n$ and let $S$ be its socle. Recall that the largest prime divisor $p$ of $n$ satisfies $(n/p)^4 < p$ (from Lemma~\ref{ntlemma}~{\bf (i)}) and that $n$ is not of the form $(q^d-1)/(q-1)$, for any prime power $q$ and any $d\geq 1$ (from Lemma~\ref{ntlemma}~{\bf (iii)}).  Then $S=A_n$ by Lemma~\ref{uniprimitiveorder}.
\end{proof}




\section{Proof of Theorems~\ref{main1} and~\ref{maintool} and Corollary~\ref{cor}}\label{sec:thrmsproof}

Every primitive group is quasiprimitive, however there are imprimitive groups that are quasiprimitive: for example the automorphism group of the line graph of the Petersen graph.  For the proof of Theorem~\ref{maintool}, we are only interested in quasiprimitive groups $G$ that are almost simple.  In this case,  $G$ is quasiprimitive if the socle of $G$ is transitive.

\begin{proof}[Proof of Theorem~$\ref{maintool}$]
Let $S$ be as in Lemma~\ref{ntlemma}, let $n$ be a finite product of distinct elements from $S$ with $n$ not prime and let $G$ be a quasiprimitive group of degree $n$.  By Lemma~\ref{primitiveAn} the only primitive groups of degree $n$ are $A_n$ or $S_n$.  We argue by contradiction and we suppose that $G\ngeq A_n$. In particular, $G$ is imprimitive. Let $\mathcal{B}$ be a non-trivial system of imprimitivity for $G$ with blocks of maximal size and let $H$ be the permutation group induced by $G$ in its action on $\mathcal{B}$. Set $r:=|\mathcal{B}|$. Observe that $H$ is primitive by our choice of $\mathcal{B}$.

As $n$ is square-free, we see that $G$ is almost simple (see for example~\cite{LiS2005}) and hence $H\cong G$.

As $H$ is primitive of degree $r$ and $r$ is a finite product of distinct elements from $S$, we get that either $r$ is prime, or $H\in \{A_r,S_r\}$ by Lemma~\ref{primitiveAn}. Suppose that $H\notin\{A_r,S_r\}$. Since $r$ is not of the form $(q^d-1)/(q-1)$, for any prime power $q$ and any $d\geq 1$, from~\cite[Theorem 1.49]{Gorenstein1982} we obtain $r=11$ and $H=\PSL(2,11)$, or $r=23$ and $H=M_{23}$. In both cases, $r$ is the largest prime divisor of $|H|=|G|$ and hence $r>(n/r)^4$. A direct inspection on $|H|$ and $n$ shows that this is impossible. Thus $H\in \{A_r,S_r\}$ and the action of $H$ on $\mathcal{B}$ is the natural action of $A_r$ or $S_r$.

From Lemma~\ref{ntlemma}, there exists a prime divisor $p$ of $n$ with $p>(n/p)^4$. Let $B\in \mathcal{B}$, let $b\in B$, and denote by $G_{\{B\}}$ the setwise stabiliser of $B$ in $G$ and by $G_b$ the stabiliser of the point $b$. As $G\cong H\in \{A_r,S_r\}$ and as $n\mid |G|$, we obtain that $p\mid r!$ and hence $p\leq r$. If $p\nmid r$, then $p>(n/p)^4\geq r^4\geq p^4$, a contradiction. Thus $p\mid r$ and hence $r>(n/r)^4$.

From above, $|G:G_{\{B\}}|=r$, $|G_{\{B\}}:G_b|=n/r$ and $G_{\{B\}}\in \{A_{r-1},S_{r-1}\}$. In particular, we have $n/r=|G_{\{B\}}:G_b|\geq r-1$. Thus $r>(n/r)^4\geq(r-1)^4$, a contradiction.
\end{proof}

\begin{proof}[Proof of Corollary~$\ref{cor}$]
Let $S$ be as in Lemma~\ref{ntlemma} and let $n$ be a finite product of distinct elements from $S$. From Lemma~\ref{ntlemma}, we have $\gcd(n,\phi(n))=1$. Now, if $n$ is prime, then  every transitive group of degree $n$ contains a cyclic transitive subgroup. If $n$ is not prime, the proof follows immediately combining Theorem~\ref{maintool} and~\cite[Theorem~4.7]{DobsonM2011} and using $\gcd(n,\phi(n))=1$.
\end{proof}

\begin{proof}[Proof of Theorem~$\ref{main1}$]Let $S$ be as in Corollary~\ref{cor}. Now, for every finite product $n$ of distinct elements from $S$, we have $\gcd(n,\phi(n))=1$ and the only quasiprimitive groups of degree $n$ are $A_n$ and $S_n$. Now~\cite[Corollary~4.10]{DobsonM2011} immediately gives that every vertex-transitive digraph of order $n$ is a Cayley digraph.
\end{proof}

\section{Comments}\label{sec:comments}
Theorem~\ref{main1} finally proves that the arithmetic structure of some Cayley numbers is much more rich than the cases that have been investigated so far. It is now interesting, in our opinion, to estimate the density of Cayley numbers. It is fairly straightforward  using the group theoretic and the number theoretic results already available in the literature to prove that Cayley numbers have density zero in the natural numbers, that is,
$$\lim_{n\to \infty}\frac{|\{m\in\mathbb{N}\mid m\textrm{ Cayley number}, m\leq n\}|}{n}=0.$$
Therefore we pose the following problem.

\begin{prob}
Determine the rate of growth of the function $$n\mapsto|\{m\in\mathbb{N}\mid m\textrm{ Cayley number}, m\leq n\}|.$$
\end{prob}

To see the importance of Corollary~\ref{cor} we need a term introduced by Hassani, Iranmanesh and Praeger~\cite{HassaniIP1998} in a paper on Cayley numbers.

Let $G$ be transitive of degree $n$.  If ${\cal B}$ and ${\cal C}$ are systems of imprimitivity for $G$ and every block of ${\cal C}$ is a union of blocks of ${\cal B}$, we write ${\cal B}\preceq{\cal C}$ (and ${\cal B}\prec{\cal C}$ when ${\cal B}\neq {\cal C}$).  Let $n = p_1^{a_1}p_2^{a_2}\cdots p_r^{a_r}$ be the prime factorisation of $n$ and define $\Omega:{\mathbb N}\setminus\{0\}\mapsto {\mathbb N}$ by $\Omega(n) = \sum_{i}^ra_i$  (thus $\Omega(n)$ is the number of prime divisors of $n$).  We say $G$ is {\bf $\Omega(n)$-step imprimitive} if there exists a sequence of systems of imprimitivity ${\cal B}_0 \prec{\cal B}_1\prec \cdots \prec{\cal B}_{\Omega(n)}$ for $G$. Observe that, if $B_{i+1}\in{\cal B}_{i+1}$ and $B_i\in{\cal B}_i$, then $\vert B_{i+1}\vert/\vert B_i\vert$ is prime for every $i\in\{1,\ldots, m - 1\}$, and that ${\cal B}_0$ and ${\cal B}_{\Omega(n)}$ are the trivial systems of imprimitivity.  If, in addition, each ${\cal B}_i$ is formed by the orbits of a normal subgroup of $G$, we say that $G$ is {\bf normally $\Omega(n)$-step imprimitive}.

\begin{thrm}[\!\!{{\cite[Theorem 1.9]{Dobson2008}}}]\label{ted}
A transitive group  of square-free degree $n$ contains a normally $\Omega(n)$-step imprimitive subgroup if and only if it contains a transitive solvable subgroup.
\end{thrm}

Hassani, Iranmanesh and Praeger~\cite[Theorem 1.1]{HassaniIP1998} showed that a vertex-transitive graph of order $pqr$, where $p$, $q$, and $r$ are distinct primes satisfying some additional arithmetic conditions, is a Cayley graph provided its automorphism group contains a normally $3$-step imprimitive subgroup (and hence a transitive solvable subgroup in view of Theorem~\ref{ted}).  The first author proved a similar result: a vertex-transitive graph of order $n$, with $\gcd(n,\phi(n))=1$ or with $\gcd(n,\phi(n))=q$ where $q$ is a prime such that $q^2\nmid (p-1)$ for every prime $p\mid (n/q)$,  is isomorphic to a Cayley graph if and only if its automorphism group contains a transitive solvable subgroup (or equivalently, is normally $\Omega(n)$-step imprimitive); see~\cite{Dobson2000a,Dobson2008}.  It is also known that any group of square-free order is solvable.  Thus the set $T$ of square-free integers for which every $2$-closed transitive group of degree $n\in T$ contains a transitive solvable subgroup is related to the set of square-free Cayley numbers, and if additional arithmetic conditions are imposed on the integers $n\in T$, the property of any $2$-closed permutation group of degree $n$ containing a transitive solvable subgroup implies $n$ is a Cayley number.  Thus the following problem is a natural one.

\begin{prob}
Determine the square-free integers $n$ with the property that every $2$-closed transitive group of degree $n$ contains a transitive solvable subgroup.
\end{prob}

We observe that it is not the case that if $n$ is a Cayley number then every transitive group of degree $n$ contains a transitive solvable subgroup.  Indeed, $14$ is a Cayley number (see~\cite[Table 1]{McKayP1994}), but $\PSL(3,2)$ has a transitive permutation  representation of degree $14$ containing no transitive solvable subgroup.
Thus, we would like to pose another problem.

\begin{prob}
Determine the square-free integers $n$ with the property that every transitive group of degree $n$ contains a transitive solvable subgroup.
\end{prob}

There is one more problem to which our work  supplies a partial solution: the problem of finding elusive groups.  This problem originated with a conjecture of Maru\v si\v c: he conjectured that the automorphism group of every vertex-transitive graph contains a semiregular element~\cite[Problem 2.4]{Marusic1981}.  Klin generalised this conjecture to the conjecture that all $2$-closed groups contain a semiregular element~\cite[Problem BCC15.12]{Cameron1997}.  Since heuristically most transitive groups contain  semiregular elements,  a transitive group which does not contain a semiregular element is called {\bf elusive}. While it is already known that every $2$-closed group of square-free degree contains a semiregular element~\cite{DobsonMMN2007}, it is not known in general if every transitive group of square-free degree is elusive.   Much is known on this topic,  see~\cite{DobsonM2011} for a list of partial results as well as for undefined terminology.  Combining Theorem~\ref{maintool} and~\cite[Corollary 4.9]{DobsonM2011} we have the following result.

\begin{cor}\label{cor2}
Let $S$ be as in Lemma~$\ref{ntlemma}$ and  let $n$ be a finite product  of distinct elements from $S$.  Then no permutation group of degree $n$ is elusive.
\end{cor}

\providecommand{\bysame}{\leavevmode\hbox to3em{\hrulefill}\thinspace}
\providecommand{\MR}{\relax\ifhmode\unskip\space\fi MR }
\providecommand{\MRhref}[2]{%
  \href{http://www.ams.org/mathscinet-getitem?mr=#1}{#2}
}
\providecommand{\href}[2]{#2}


\begin{thebibliography}{10}

\bibitem{Cameron1997}
P.~J.~ Cameron, Problems from the {S}eventeenth {B}ritish
  {C}ombinatorial {C}onference, \textit{Discrete Math.} \textbf{167/168} (1997),
  605--615.


\bibitem{Dobson2000}
E.~Dobson, Classification of vertex-transitive graphs of order a
  prime cubed. {I}, \textit{Discrete Math.} \textbf{224} (2000), no.~1-3, 99--106.

\bibitem{Dobson2000a}
E.~Dobson, On solvable groups and circulant graphs, \textit{European J. Combin.}
  \textbf{21} (2000), no.~7, 881--885.

\bibitem{Dobson2008}
E.~Dobson, On solvable groups and {C}ayley graphs, \textit{J. Combin. Theory Ser. B} \textbf{98} (2008), no.~6, 1193--1214.

\bibitem{DobsonMMN2007}
E.~Dobson, A.~Malni{\v{c}}, D.~Maru{\v{s}}i{\v{c}},
  L.~A.~Nowitz, Minimal normal subgroups of transitive permutation
  groups of square-free degree, \textit{Discrete Math.} \textbf{307} (2007), no.~3-5,
  373--385.

\bibitem{DobsonM2011}
E.~Dobson, D.~Maru{\v{s}}i{\v{c}}, On semiregular elements of
  solvable groups, \textit{Comm. Algebra} \textbf{39} (2011), no.~4, 1413--1426.

\bibitem{Gorenstein1982}
D.~Gorenstein, \emph{Finite simple groups}, University Series in
  Mathematics, Plenum Publishing Corp., New York, 1982.

\bibitem{HassaniIP1998}
A.~Hassani, M.~A.~Iranmanesh, C.~E.~Praeger, On
  vertex-imprimitive graphs of order a product of three distinct odd primes,
  \textit{J. Combin. Math. Combin. Comput.} \textbf{28} (1998), 187--213.

\bibitem{IranmaneshP2001}
M.~A.~Iranmanesh, C.~E.~Praeger, On non-{C}ayley
  vertex-transitive graphs of order a product of three primes, \textit{J. Combin.
  Theory Ser. B} \textbf{81} (2001), no.~1, 1--19.


\bibitem{LiS2005}
C.~H.~Li, {\'A}.~Seress, On vertex-transitive non-{C}ayley graphs of square-free order,
\textit{Des. Codes Cryptogr.} \textbf{34} (2005), no.~2-3, 265--281.


\bibitem{LiebeckS1985a}
M.~W.~Liebeck, J.~Saxl, Primitive permutation groups containing
  an element of large prime order, \textit{J. London Math. Soc. (2)} \textbf{31}
  (1985), no.~2, 237--249.

\bibitem{Marusic1981}
D.~Maru{\v{s}}i{\v{c}}, On vertex symmetric digraphs, \textit{Discrete Math.}
  \textbf{36} (1981), no.~1, 69--81.

\bibitem{Marusic1983}
D.~Maru{\v{s}}i{\v{c}}, Cayley properties of vertex symmetric graphs,
\textit{Ars Combin.}  \textbf{16} (1983), 297--302.

\bibitem{McKayP1994}
B.~D.~McKay, C.~E.~Praeger, Vertex-transitive graphs which
  are not {C}ayley graphs. {I}, \textit{J. Austral. Math. Soc. Ser. A} \textbf{56}
  (1994), no.~1, 53--63.

\bibitem{McKayP1996}
B.~D.~McKay, C.~E.~Praeger, Vertex-transitive graphs that are not {C}ayley graphs. {II}, \textit{J.
  Graph Theory} \textbf{22} (1996), no.~4, 321--334.

\bibitem{Sabidussi1958}
G.~Sabidussi, On a class of fixed-point-free graphs, \textit{Proc. Amer. Math. Soc.}, \textbf{9} (1958), 800--804.

\bibitem{Seress1998}
{\'A}.~Seress, On vertex-transitive, non-{C}ayley graphs of order
  {$pqr$}, \textit{Discrete Math.} \textbf{182} (1998), no.~1-3, 279--292.

\bibitem{Serre}J-P.~Serre, \textit{A course in arithmetic},  Springer-Verlag, New York, 1973.
\end{thebibliography}

\end{document}